\documentclass[12pt]{amsart}
\usepackage{amssymb}
\newtheorem{theorem}{Theorem}[section]
\newtheorem{lemma}[theorem]{Lemma}
\newtheorem{prop}[theorem]{Proposition}
\newtheorem{cor}[theorem]{Corollary}
\theoremstyle{definition}
\newtheorem{definition}[theorem]{Definition}

\theoremstyle{remark}
\newtheorem{remark}[theorem]{Remark}

\numberwithin{equation}{section}

\newcommand\B{\mathbb{B}}
\newcommand\C{\mathbb{C}}

\newcommand\R{\mathbb{R}}

\newcommand\cA{\mathcal{A}}

\newcommand\cG{\mathcal{G}}

\newcommand\cK{\mathcal{K}}

\newcommand\cJ{\mathcal{J}}

\newcommand\cS{\mathcal{S}}

\newcommand\cU{\mathcal{U}}

\newcommand\RE{\mathrm{Re}\,}
\newcommand\IM{\mathrm{Im}\,}
\newcommand\EXP{\mathrm{Exp}\,}
\newcommand\Ran{\mathrm{Ran}\,}
\newcommand\Ker{\mathrm{Ker}\,}
\newcommand\spa{\mathrm{span}\,}
\newcommand\RAU{R\cA\cU}
\newcommand\IAU{I\cA\cU}

\newcommand\inpr[2]{\langle{#1,#2}\rangle}

\begin{document}
\title{Every sum system is divisible}

\author{MASAKI IZUMI}
\address{Department of Mathematics\\ 
Kyoto University, Kyoto, Japan.}
\email{izumi@math.kyoto-u.ac.jp}



\thanks{Work supported by JSPS}
\begin{abstract} 
We show that every sum system is divisible. 
Combined with B. V. R. Bhat and R. Srinivasan's result, this shows that 
every product system arising from a sum system (and every generalized CCR flow) is either of type I or type III. 
A necessarily and sufficient condition for such a product system to be of type I is obtained.  
\end{abstract}

\maketitle

\section{Introduction}\label{intro} 
An $E_0$-semigroup is a weakly continuous semigroup of unital $*$-endomorphisms on $\B(H)$, 
the algebra of all bounded operators on a separable infinite dimensional Hilbert space $H$. 
W. Arveson \cite{Arv1} introduced the notation of a product system, 
(a continuous tensor \textit{product system} of Hilbert spaces), and showed that 
the product system associated with an $E_0$-semigroup completely determines the cocycle conjugacy class 
of the $E_0$-semigroup. 
On the other hand, Arveson \cite{Arv2} also showed that every product system arises from an 
$E_0$-semigroup (see also \cite{L}, \cite{Sk}). 

$E_0$-semigroups (and hence product systems) are classified into three categories, type I, II, and III. 
Using a quasi-free representation of the CAR (canonical anticommutation relation) algebras, 
R. Powers \cite{Po1} constructed the first example of an $E_0$-semigroup of type III. 
It is technically very difficult to construct such an example, and one had to wait for more than 10 years 
before B. Tsirelson \cite{T1} produced uncountably many mutually non-isomorphic product systems of type III. 
His construction uses continuous sums of Hilbert spaces coming from off white noises. 

Recently, a few attempts \cite{BS}, \cite{I}, \cite{IS} were  made to understand Tsirelson's construction from the view point of 
functional analysis. 
Bhat and Srinivasan \cite{BS} introduced the notion of a sum system, which is an axiomatization 
of Tsirelson's continuous sum of Hilbert spaces. 
A sum system gives rise to a product system via the second quantization procedure using the CCR 
(canonical commutation relation). 
Among the others, a dichotomy result about types was proved in \cite{BS}, which says that 
every product system arising from a divisible sum system is either of type I or of type III. 
A sum system is said to be divisible if it has sufficiently many real and imaginary addits 
(called additive units in \cite{BS}). 
While Bhat and Srinivasan \cite{BS} dealt with product systems, the present author \cite{I} directly gave a description of Tsirelson's 
$E_0$-semigroups in terms of perturbations of the shift semigroup of $L^2(0,\infty)$ and the CCR algebras. 

As a consequence of the two approaches \cite{BS} and \cite{I}, a class of $E_0$-semigroup, called generalized CCR flows, was 
introduced in \cite{IS}. 
A generalized CCR flow is constructed from a pair of $C_0$-semigroups acting on a real Hilbert space on one hand, and on one hand, 
the product system associated with it arises from a sum systems. 
In \cite{IS}, we constructed continuously many mutually non-cocycle conjugate generalized CCR flows of type III different from 
Tsirelson's examples 

The main purpose of this paper is to show that every sum system is divisible. 
Combined with the above mentioned dichotomy result, this implies that every product system arising from a sum system 
(and every generalized CCR flow) is either of type I or type III. 
In our proof, we carefully analyze the domains of the generators of the two $C_0$-semigroups giving the generalized CCR flows, 
whose product system arises from the given sum system. 
In \cite{IS}, a type I criterion was given for a divisible sum system of finite index. 
Our main theorem shows that the spaces of real and imaginary addits have nice subspaces with manageable topologies, which 
enables us to prove the type I criterion in full generality. 

The author would like to thank R. Srinivasan for useful discussions. 
%%%%%%%%%%%%%%%%%%%%%%%%%%%%%%%%%%%%%%%%%%%%%%%%%%%%%%%%%%%%%%%%%%%%%%%%%%%%%%%%%%%%%%%%%%%%%%%%%%%%%%%%%%%%%%
%%%%%%%%%%%%%%%%%%%%%%%%%%%%%%%%%%%%%%%%%%%%%%%%%%%%%%%%%%%%%%%%%%%%%%%%%%%%%%%%%%%%%%%%%%%%%%%%%%%%%%%%%%%%%%
%%%%%%%%%%%%%%%%%%%%%%%%%%%%%%%%%%%%%%%%%%%%%%%%%%%%%%%%%%%%%%%%%%%%%%%%%%%%%%%%%%%%%%%%%%%%%%%%%%%%%%%%%%%%%%
\section{Preliminaries and notation}\label{pre}
In this section we fix the notation used in this paper. 
For an operator $A$, we denote the range of $A$ by $\Ran(A)$, and the kernel of $A$ by $\Ker(A)$.  
We denote the identity operator on a Hilbert space $H$ by $I_H$ (or simply by $I$ if no confusion arises).  
Our inner product is linear in the first variable. 

%%%%%%%%%%%%%%%%%%%%%%%%%%%%%%%%%%%%%%%%%%%%%%%%%%%%%%%%%%%%%%%%%%%%%%%%%%%%%%%%%%%%%%%%%%%%%%%%%%%%%%
\subsection{$E_0$-semigroups and product systems}
For $E_0$-semigroups, our basic reference is \cite{Arv3}. 

\begin{definition}\label{productsystem}
A product system of Hilbert spaces is an one-parameter family
of separable complex Hilbert spaces $\{H_t;t \geq 0 \}$, together with unitary operators 
$U_{s,t} : H_s \otimes H_t \rightarrow H_{s+t}$ for  $s, t \in (0,\infty)$, 
satisfying the following two axioms: 
\begin{itemize}
\item [(1)] (Associativity) For any $s_1, s_2, s_3 \in (0,\infty)$
$$U_{s_1, s_2 + s_3}( I_{H_{s_1}} \otimes U_{s_2 ,
s_3})=  U_{s_1+ s_2 , s_3}( U_{s_1 ,
s_2} \otimes I_{H_{s_3}}).$$

\item[(2)] (Measurability) There exists a countable set $H^0$ of
sections $\R\ni t \rightarrow h_t \in H_t$ such that $ t  \mapsto 
\langle h_t, h_t^\prime\rangle$ is measurable for any two $h, h^\prime \in H^0$, and the set $\{h_t; h \in H^0\}$ is total 
in $H_t$, for each $ t \in (0,\infty)$. 
Further it is also assumed that the map $(s,t) \mapsto \langle U_{s,t}(h_s \otimes h_t), h^\prime_{s+t} \rangle$ is measurable 
for any two $h , h^\prime \in H^0$.
\end{itemize}

Two product systems $(\{H_t\}, \{U_{s,t}\})$ and $(\{H^\prime_t\}, \{U^\prime_{s,t}\})$
are said to be isomorphic if there exists a unitary operator $V_t:H_t \rightarrow H_t^\prime$, 
for each $t \in (0,\infty)$ satisfying $$V_{s+t}U_{s,t}= U_{s,t}^\prime (V_s \otimes V_t).$$
\end{definition}

The above definition is slightly different from Arveson's original one though the two definitions are equivalent, of course. 
Note that we need not consider measurable structures while dealing with isomorphisms of product systems (see \cite{L}). 

For an $E_0$-semigroup $\{\alpha_t\}_{t \geq 0}$, let $H_t$ be the space of intertwiners between the identity of $H$ and $\alpha_t$, that is,  
$$H_t =\{T \in \B(H); \; \alpha_t(X) T = TX, \; \forall X \in\B(H)\}.$$ 
Then $H_t$ is a Hilbert space with inner product $\inpr{S}{T} I_H= T^*S$. 
Identifying $x_1\otimes x_2$ with $x_1x_2$ for $x_1\in H_t$ and $x_2\in H_s$, one can see that the family $\{H_t\}_t$ satisfies 
the axiom of a product system. 

\begin{definition}\label{unit}
A unit for a product system is a non-zero section $\{u_t;t \geq 0 \}$, such that the map $t \mapsto \langle u_t, h_t\rangle$ is
measurable for any $h \in H^0$ and 
$$U_{s,t}(u_s \otimes u_t)= u_{s+t}, ~  \forall s,t \in (0,\infty).$$
\end{definition}

A product system (an $E_0$-semigroup) is said to be of type I, if units exists for the product system and 
they generate the product system  (see \cite{Arv1}, \cite{Arv3} for the precise meaning). 
It is of type II if units exist but they do not generate the product system. 
It is of type III (or unitless) if there is no unit. 

%%%%%%%%%%%%%%%%%%%%%%%%%%%%%%%%%%%%%%%%%%%%%%%%%%%%%%%%%%%%%%%%%%%%%%%%%%%%%%%%%%%%%%%%%%%%%%%%%%%%%%
\subsection{Symmetric Fock space}
For a complex Hilbert space $K$, we denote by $\Gamma(K)$ the symmetric Fock space associated with $K$, 
and by $\Phi$ the vacuum vector in $\Gamma(K)$ (see \cite{Pa}). 
For any $x \in K$, the exponential vector of $x$ is defined by $$e(x)= \oplus_{n=0}^\infty  
\frac{x^{\otimes^n}}{\sqrt{n!}},$$ where $x^{\otimes^0}=\Phi$.  
Then the set of all exponential vectors 
$\{e(x): x \in K\}$ is a linearly independent total set in $\Gamma(K)$. 
For a unitary $U\in \B(K)$, we denote by $\EXP(U)$ the unitary in $\B(\Gamma(K))$ 
given by $\EXP(U)e(x)=e(Ux)$. 
The Weyl operator $W(x)$ for $x \in K$ is a unitary operator of $\Gamma(K)$ determined by 
$$W(x)(e(y))= e^{-\frac{\|x\|^2}{2}- \inpr{y}{x}} e(y+x).$$ 
The Weyl operators satisfy the relation $W(x)W(y)=e^{i\IM \inpr{x}{y}}W(x+y)$. 
The $*$-algebra generated by $\{W(x)\}_{x\in K}$ is called the Weyl algebra for $K$. 

For a real Hilbert space $G$, we denote  the complexification of $G$ by $G^{\C}$. 
For two Hilbert spaces $G_1$, $G_2$, we denote by $\cS(G_1,G_2)$ the set of bounded invertible operators $A$ in $\B(G_1, G_2)$ 
such that $I-A^*A$ is a Hilbert-Schmidt operator.  
An operator in $\cS(G_I,G_2)$ is called an equivalence operator. 
In the above definition and elsewhere, by invertibility we mean that the inverse is also bounded. 

For two real Hilbert spaces $G_I, G_2$ and
$A \in \cS(G_1,G_2 )$, define
a real liner operator $S_A:G_1^{\C} \rightarrow
G_2^{\C}$ by
$S_A(u+iv)=Au+i(A^{-1})^*v$ for $u,v \in
G_1$. Then $S_A$ is a symplectic
isomorphism between $G_1^\C$ and
$G_2^\C$
(i.e. $S_A$ is a real linear, bounded,  invertible map satisfying $\IM(\langle S_Ax,S_Ay\rangle
)=\IM\langle x,y\rangle $ for all $x,y
\in G_1^\C$, see \cite[page 162]{Pa}). In general $S_A$ is not complex linear, unless $A$ is unitary.

The following theorem, a generalization of Shales theorem, is used to 
construct the product system from a sum system in \cite{BS}.

\begin{theorem}\label{Shale}
Let $G_1$ and $G_2$ be real Hilbert spaces and $A \in \cS(G_1,G_2)$. 
Then there exists a unique unitary operator $\Gamma(A):\Gamma(G_1^\C) \rightarrow \Gamma(G_2^\C)$ satisfying
$$\Gamma(A)W(z){\Gamma(A)}^* =W(S_Az),\quad  \forall z\in G_1^{\C},$$
with the normalization condition $\inpr{\Gamma(A) \Phi_1}{\Phi_2}\in \R^{+}$, where $\Phi_1$ and $\Phi_2$ are 
the vacuum vectors in $\Gamma(G_1^\C)$ and $\Gamma(G_2^\C)$ respectively. 
If moreover $G_3$ is a real Hilbert space and $B \in \cS(G_2, G_3)$, then
$\Gamma(BA)=\Gamma(B)\Gamma(A)$ holds. 
\end{theorem}
Note that $\Gamma(A^{-1})={\Gamma(A)}^*$ automatically holds from the above statement. 

The next lemma follows from the proof of \cite[Lemma 4.6]{IS}: 

\begin{lemma}\label{bdd}
Let $G_1$ and $G_2$ be real Hilbert spaces and let $R_0$ be a real linear operator on from $G_1^\C$ to $G_2^\C$ with dense domain $D(R_0)$, 
which preserves imaginary parts of the inner product. 
Suppose that there exists a unitary operator $U \in \B(\Gamma(G_1^\C),\Gamma(G_2^\C))$ satisfying 
\begin{eqnarray}\label{intertwine} UW(x)U^*& = & W(R_0x),~~\forall~x \in D(R_0).\end{eqnarray} 
Then $R_0$ extends to a bounded invertible operator $R$ from $G_1^\C$ onto $G_2^\C$. 
Further it is true that $R \in \cS(G_1 \oplus G_1, G_2 \oplus G_2)$, where we identify $G_i\oplus G_i$ with $G_i^\C$ equipped with 
the real inner product $\inpr{\cdot}{\cdot}_\R=\RE\inpr{\cdot}{\cdot}$.  
\end{lemma}

%%%%%%%%%%%%%%%%%%%%%%%%%%%%%%%%%%%%%%%%%%%%%%%%%%%%%%%%%%%%%%%%%%%%%%%%%%%%%%%%%%%%%%%%%%%%%%%%%%%%%%
\subsection{Sum systems} 
Next we define the notion of a sum system. 

\begin{definition}\label{sumsystem} 
A sum system is a two-parameter family $\{G_{s,t}\}_{0 \leq s < t \leq \infty}$ of closed subspaces of a real Hilbert space 
$G_{0,\infty}$ satisfying the inclusion relations $G_{s,t} \subset G_{s^{\prime}, t^{\prime}}$ for all 
$(s,t)\subset (s^{\prime}, t^{\prime})$ together with a $C_0$-semigroup $\{S_t\}_{t\geq 0}$ acting on $G_{0, \infty}$ such that 
the following hold for any 
$s\in (0,\infty)$ and $t\in (0,\infty]$ with $s<t$: 
\begin{itemize}
\item [(1)] The restriction of $S_s$ to $G_{0,t}$ is in $\cS(G_{0,t}, G_{s, s+t})$. 
\item [(2)] If $A_{s,t}:G_{0,s}\oplus G_{s,s+t}\mapsto G_{0,s+t}$, is  the map
$A_{s,t}(x\oplus y)=x+y$, for $x \in G_{0,s}$ and $y \in G_{s,s+t}$, then 
$A_{s,t} \in \cS(G_{0,s}\oplus G_{s,s+t}, G_{0,s+t})$.  
\end{itemize}
We say that two sum systems $(\{G_{a,b}\},\{S_t\})$ and $(\{G'_{a,b}\},\{S'_t\})$ are isomorphic 
if there exists a family of equivalence operators $U_t\in \cS(G_{0,t},G'_{0,t})$ preserving 
every structure of the sum systems, i.e, 
$\{U_t\}$ satisfies $$ A^\prime_{s,t}(U_s \oplus S^\prime_s|_{G_{0,t}^\prime} U_t) = 
U_{s+t}A_{s,t}(I_{G_{0,s}}\oplus S_s|_{G_{0,t}}),$$ where $A_{s,t}^\prime$ is defined  for $(\{G'_{a,b}\},\{S'_t\})$ 
similar to $A_{s,t}$ above.
\end{definition}

The above definition, adopted in \cite{IS}, is slightly stronger than the one given in \cite{BS}. 

Given a sum system $(\{G_{s,t}\},\{S_t\})$, we define Hilbert spaces $H_t=\Gamma(G_{0,t}^\C)$, and  unitary operators
$U_{s,t}:H_s\otimes H_t \mapsto H_{s+t},$ by
$U_{s,t}=\Gamma(A_{s,t})(I_{H_s}\otimes \Gamma(S_s|_{G_{0,t}})).$ It is
proved in \cite{BS} that $(\{H_t\}, \{U_{s,t}\})$ forms a product system. 
Isomorphic sum systems give rise to isomorphic product systems.

%%%%%%%%%%%%%%%%%%%%%%%%%%%%%%%%%%%%%%%%%%%%%%%%%%%%%%%%%%%%%%%%%%%%%%%%%%%%%%%%%%%%%%%%%%%%%%%%%%%%%%
\subsection{GCCR flows}
Let $G$ be a real Hilbert space, and let $\{S_t\}$ be the shift semigroup of $L^2((0,\infty),G)$ (and of $L^2((0,\infty),G^\C)$) 
defined by 
$$(S_tf)(x)=\left\{
\begin{array}{ll}0& (x<t)\\
f(x-t) &(t\leq x)
\end{array}
\right. ,$$
for $f \in L^2((0,\infty),G).$ 
The CCR flow of index $\dim G$ is an $E_0$-semigroup $\alpha$ acting on $\B(\Gamma(L^2((0,\infty),G^\C)))$, which is 
determined by $\alpha_t(W(f))=W(S_tf)$ for $f\in L^2((0,\infty),G^\C)$. 
The corresponding product system is the exponential product system of index $\dim G$ (see \cite{Arv2}), 
which is the simplest example of a product system arising from a sum system.   
Namely, let $G_{s,t}=L^2((s,t),G)$. 
Then the sum system $(\{G_{s,t}\},\{S_t\})$, which we call the \textit{shift sum system of index $\dim K$}, 
gives rise to the exponential sum system. 

Since the CCR relation involves only the imaginary part of the inner product of the test functions, 
to construct an $E_0$-semigroup, we do not necessarily use the same time translation for both the real and imaginary test functions. 

\begin{definition}\label{perturb} Let $\{S_t\}$ and $\{T_t\}$ be $C_0$-semigroups acting on a real Hilbert space $G$. 
We say that $\{T_t\}$ is a perturbation of $\{S_t\}$, if they satisfy,
\begin{enumerate}
\item[(1)] ${T_t}^*S_t =I.$
\item[(2)] $S_t -T_t$ is a Hilbert-Schmidt operator.
\end{enumerate}
We call $(\{S_t\},\{T_t\})$ as above a \textit{perturbation pair}. 
\end{definition}

Let $(\{S_t\}, \{T_t\})$ be a perturbation pair of $C_0$-semigroups acting on $G$. 
Then an easy application of well-known criteria \cite{Ara3},\cite{vD} shows the existence of 
the generalize CCR flow for $(\{S_t\}, \{T_t\})$ defined below (see \cite[Lemma 2.3]{IS} for the proof). 

\begin{definition} Given a perturbation pair $(\{S_t\}, \{T_t\})$ of $C_0$-semigroups acting on a real Hilbert space $G$, 
we say that the $E_0$-semigroup $\{\alpha_t\}_{t\geq 0}$ acting on 
$\B(\Gamma(G^\C))$ determined by 
$$\alpha_t(W(x+iy))=W(S_tx+iT_ty),\quad x,y\in G$$
is a generalized CCR flow associated with the pair $\{S_t\}$ and $\{T_t\}$. 
\end{definition}

%%%%%%%%%%%%%%%%%%%%%%%%%%%%%%%%%%%%%%%%%%%%%%%%%%%%%%%%%%%%%%%%%%%%%%%%%%%%%%%%%%%%%%%%%%%%%%%%%%%%%%
\subsection{From sum systems to perturbation pairs}\label{stop}
In the remaining two subsections, we recall the description of the $E_0$-semigroup associated with the product
system constructed out of a sum system given in \cite[Section 3]{IS}. 

Let $(\{G_{a,b}\}, \{S_t\})$ be a sum system. 
Denote $G =G_{0,\infty}$, $A_t = A_{t, \infty}$. 
We may consider $S_t$ as a bounded linear invertible map from $G$ onto $G_{t, \infty}$. 
Hence $(S_t^*)^{-1}$ is a well-defined bounded operator from $G$ onto $G_{t, \infty}$. 
When there is no possibility of confusion, we sometimes regard $(S_t^*)^{-1}$ as an element of $\B(G)$. 
Define $T_t \in \B(G)$, by 
$$T_t =(A_t^*)^{-1} A_t^{-1}(S_t^*)^{-1}~~ \forall~t \in [0,\infty).$$
Then $\{T_t\}$ forms a $C_0$-semigroup on $G$ (\cite[Lemma 3.2]{IS}). 

The following Proposition is proved in \cite[Porposition 3.3, Corollary 3.4]{IS}: 
\begin{prop} \label{e0} Let $(\{G_{s,t}\},\{S_t\})$ be a sum system, and let $\{T_t\}$ be the $C_0$-semigroup as above. 
Then $(\{S_t\},\{T_t\})$ is a perturbation pair of $C_0$-semigroups. 
The product system for the generalize CCR flow for  $(\{S_t\},\{T_t\})$ is the one constructed out of 
the sum system  $(\{G_{s,t}\},\{S_t\})$.
\end{prop}

We say that the above pair $(\{S_t\},\{T_t\})$ of $C_0$-semigroups is associated with the sum system 
$(\{G_{a,b}\},\{S_t\})$. 

%%%%%%%%%%%%%%%%%%%%%%%%%%%%%%%%%%%%%%%%%%%%%%%%%%%%%%%%%%%%%%%%%%%%%%%%%%%%%%%%%%%%%%%%%%%%%%%%%%%%%%
\subsection{From perturbation pairs to sum systems}\label{ptos}
Let $G$ be a real Hilbert space, and let $(\{S_t\},\{T_t\})$ be a perturbation pair of $C_0$-semigroup acting on $G$. 
We describe the product system for the generalized CCR flow associated with the pair $(\{S_t\}, \{T_t\})$. 

Let $G_{0,t} = \Ker(T_t^*)$, $G_{a,b}=S_a(G_{0,b-a})$, and 
$$G_{0,\infty}=\overline{\bigcup_{t > 0} G_{0,t}}. $$ 
Let $P:G \rightarrow G_{0,\infty}$ be the orthogonal projection. 
We define $S^0_t$ and $T^0_t$ by 
$$S^0_t= PS_tP,\quad T^0_t= PT_tP.$$

The following is \cite[Proposition 3,6]{IS}: 

\begin{prop}\label{StTt} Let $G$ be a real Hilbert space and let $(\{S_t\},\{T_t\})$ be a perturbation pair of 
$C_0$-semigroups acting on $G$. 
Let $\{G_{s,t}\}$, $\{S_t^0\}$, and $\{T_t^0\}$ be as above. Then 
\begin{itemize}
\item[$(1)$] The system $(\{G_{a,b}\}, \{S_t^0\})$ forms a sum system. \medskip
\item[$(2)$] The pair of $C_0$-semigroups 
$(\{S_t^0\},\{T_t^0\})$ is associated with $(\{G_{a,b}\}, \{S_t^0\})$. 
In consequence, the product system for the generalized CCR flow arising from 
$(\{S_t^0\},\{T_t^0\})$ is isomorphic to the one arising from $(\{G_{a,b}\},\{S_t^0\})$.\medskip 
\item[$(3)$] 
The product system for the generalized CCR flow arising from $(\{S_t\}, \{T_t\})$ is 
isomorphic to the product system arising from $(\{G_{a,b}\}, \{S_t^0\})$. 
In consequence, the generalized CCR flow arising from the pair $(\{S_t\},\{T_t\})$ is cocycle conjugate to that arising 
from $(\{S^0_t\}, \{T_t^0\})$.
\end{itemize}
\end{prop}

The following easy observation \cite[Lemma 3.5]{IS} is very useful:

\begin{lemma} \label{direct sum} Let $(\{S_t\},\{T_t\})$ be a perturbation pair of $C_0$-semigroups acting on 
a real Hilbert space $G$. 
The two operators $S_tT_t^*$ and $I-S_tT_t^*$ are idempotents such that 
$\Ran(S_tT_t^*)=\Ker(I-S_tT_t^*)=\Ran(S_t)$ and  
$\Ran(I-S_tT_t^*)=\Ker (S_tT_t^*)=G_{0,t}$. 
In particular, the Hilbert space $G$ is a topological direct sum of $G_{0,t}$ and $\Ran(S_t)$.  
\end{lemma}
%%%%%%%%%%%%%%%%%%%%%%%%%%%%%%%%%%%%%%%%%%%%%%%%%%%%%%%%%%%%%%%%%%%%%%%%%%%%%%%%%%%%%%%%%%%%%%%%%%%%%%%%%%%%%%%%%%%%%%%%%%%%%
%%%%%%%%%%%%%%%%%%%%%%%%%%%%%%%%%%%%%%%%%%%%%%%%%%%%%%%%%%%%%%%%%%%%%%%%%%%%%%%%%%%%%%%%%%%%%%%%%%%%%%%%%%%%%%%%%%%%%%%%%%%%%
%%%%%%%%%%%%%%%%%%%%%%%%%%%%%%%%%%%%%%%%%%%%%%%%%%%%%%%%%%%%%%%%%%%%%%%%%%%%%%%%%%%%%%%%%%%%%%%%%%%%%%%%%%%%%%%%%%%%%%%%%%%%%
\section{Divisibility}
Bhat-Srinivasan \cite{BS} introduced the notion of divisibility for sum systems.      
In this section we show that every sum system (in the sense of Definition \ref{sumsystem}) is divisible. 

\subsection{Addits}
We first recall the definition of addits for a sum system, which were called additive units in \cite{BS}. 

\begin{definition} Let $(\{G_{a,b}\}, \{S_t\})$ be a sum system. 
A real addit for the sum system $(\{G_{a,b}\}, \{S_t\})$ is a
family $\{x_t\}_{t \in (0,\infty)}$ such that
$x_t \in G_{0,t}$  for  $\forall t \in (0, \infty)$, satisfying the following conditions: 
\begin{itemize}
\item [(1)] The map $ t \mapsto \inpr{x_t}{x}$ is measurable for any $ x \in G_{0,\infty}$. 
\item [(2)] The cocycle relation $x_s+ S_s x_t = x_{s+t},$ holds for $\forall s, t\in (0, \infty)$. 
\end{itemize}
An imaginary addit for the sum system $(\{G_{a,b}\}, \{S_t\})$ is a family $\{y_t\}_{t \in (0,\infty)}$ such that
$y_t \in G_{0,t}$ for $\forall ~t \in (0,  \infty)$, satisfying the following conditions: 
\begin{itemize}
\item [(1)] The map $ t \mapsto \inpr{y_t}{y}$ is measurable for  $\forall y \in G_{0,\infty}$. 
\item [(2)] The cocycle relation $(A_{s,t}^*)^{-1}(y_s \oplus (S_s^*)^{-1} y_t) =y_{s+t}$ holds for $\forall s, t,\in (0, \infty).$ 
\end{itemize}
We denote by $\RAU$ the set of all real addits and by $\IAU$ the set of all imaginary addits respectively, which are real linear spaces. 
\end{definition}

For a given real addit $\{x_t\}$, set $x_{s,t}= S_s(x_{t-s}) \in G_{s, t}$. 
Similarly for a given imaginary addit $\{y_t\}$, set $y_{s,t}=(S_s^*)^{-1}(y_{t-s}) \in G_{s, t}$.
For $(s_1, s_2) \subset (0,s)$, we also set 
$$G_{0,s} \ni {}^s y^\prime_{s_1,s_2}= (\tilde{A}^*)^{-1}(0 \oplus y_{s_1, s_2} \oplus 0),$$ 
where $\tilde{A}$ is the map $G_{0,s_1} \oplus G_{s_1,s_2} \oplus G_{s_2, s}\ni x\oplus y \oplus z \mapsto x+y+z \in G_{0,s}$. 
By definition, we have ${ }^sy^\prime_{s_1,s_2} \in \left(G_{0,s_1} \bigvee G_{s_2,s}\right)^{\perp} \cap G_{0,s}.$ 

\begin{definition} A sum system $(\{G_{a,b}\},\{S_t\})$ is said to be 
divisible if each of $\RAU$ and $\IAU$ generates the sum system, that is, 
$$G_{0,s} =  \overline{\spa_\R  \{x_{s_1,s_2}; (s_1,s_2) \subseteq (0,s), \{x_t\} \in\RAU\}},$$ 
$$G_{0,s} =  \overline{\spa_\R  \{{}^s y^\prime_{s_1,s_2}; (s_1,s_2) \subseteq(0,s), \{y_t\} \in \IAU\}}.$$
\end{definition}

\begin{remark}
Let $(\{H_t\},\{U_{s,t}\})$ be a product system arising from $(\{G_{s,t}\},\{S_t\})$. 
Bhat-Srinivasan \cite{BS} observed that every real addit $\{x_t\}$ gives a unitary operator $W(x_t)$ on $H_t$, and the family 
$\{W(x_t)\}_{t>0}$ forms an automorphism of the product system. 
For an imaginary addit $\{y_t\}$, the family $\{W(iy_t)\}_{t>0}$ forms an automorphism as well. 
Using this observation, they showed the following dichotomy result:  
\end{remark}

\begin{theorem}[\cite{BS}] \label{IorIII} Every product system arising from a divisible sum system is either type I or type III. 
\end{theorem}

The following is \cite[Proposition 37 (ii)]{BS}, \cite[Proposition 4.3]{IS}. 

\begin{prop}\label{h}
Let $(\{G_{(a,b)}\},\{S_t\})$ be a sum system. If $\{x_t\} \in \RAU$ and $\{y_t\}
\in \IAU$, then $$\langle x_t, y_t\rangle = \langle x_1, 
y_1\rangle t , ~~\forall~ t \in (0, \infty).$$ In general for any two intervals 
$(s_1,s_2),~(t_1, t_2) \subset (0, s)$, it is true that
\begin{eqnarray*} 
\langle x_{s_1, s_2}, {}^s y^\prime_{t_1,t_2} \rangle =
\langle x_1, y_1 \rangle |
(s_1, s_2) \cap (t_1, t_2)|,
\end{eqnarray*}
where $|.|$ is the Lebesgue
measure on
$\R$.
\end{prop}

%%%%%%%%%%%%%%%%%%%%%%%%%%%%%%%%%%%%%%%%%%%%%%%%%%%%%%%%%%%%%%%%%%%%%%%%%%%%%%%%%%%%%%%%%%%%%%%%%%%%%%%%%%%%%%%%%%%%
\subsection{Additive cocycles} 
Throughout this subsection, we assume that $G$ is a real Hilbert space, and that $(\{S_t\}, \{T_t\})$ is a perturbation pair of 
$C_0$-semigroups acting on $G$. 

\begin{definition} 
A real additive cocycle for the pair $(\{S_t\},\{T_t\})$ is a family $\{c_t\}_{t \geq 0}$ such that
$c_t \in \Ker(T_t^*)$ for $\forall t \in (0, \infty)$, satisfying the following conditions: 
\begin{itemize}
\item [(1)] The map $ t \mapsto \langle c_t, x\rangle$ is measurable for
any $ x \in G$.
\item [(2)] The cocycle relation $c_s+ S_s c_t = c_{s+t}$ holds for $\forall s, t\geq 0$. 
\end{itemize}

An imaginary additive cocycle for the pair $(\{S_t\},\{T_t\})$ is a 
family $\{d_t\}_{t \geq 0}$ such that
$d_t \in \Ker(S_t^*)$ for $\forall ~t \in (0, \infty)$, satisfying the following conditions:
\begin{itemize}
\item [(1)] The map $ t \mapsto \langle y_t, y\rangle$ is measurable for
any $ y \in G$.
\item [(2)] The cocycle relation $d_s +T_s d_t =d_{s+t}$ holds for $\forall s, t\geq 0.$
\end{itemize}
\end{definition}

\begin{remark}
Let $\{\alpha\}_{t\geq 0}$ be the generalized CCR flow associated with $(\{S_t\},\{T_t\})$. 
Then every real additive cocycle $\{c_t\}_{t\geq 0}$ gives a family of unitary operators $\{W(c_t)\}_{t\geq 0}$ on $\Gamma(G^\C)$, 
which forms a gauge cocycle, that is, the unitary $W(x_t)$ is in the commutant of the image of $\alpha_t$ and the cocycle relation 
$W(x_{s+t})=W(x_s)\alpha_t(W(x_t))$ holds for $s,t>0$. 
For an imaginary additive cocycle $\{d_t\}_{t\geq 0}$, the family $\{W(id_t)\}_{t\geq 0}$ is a gauge cocycle as well. 
\end{remark}

In a similar way as in the proof of Proposition \ref{h}, we can show 

\begin{prop}\label{hh} Let the notation be as above. 
If $\{c_t\}$ is a real additive cocycle and $\{d_t\}$ is an imaginary 
additive cocycle for $(\{S_t\},\{T_t\})$, then 
$$\langle c_t, d_t\rangle = \langle c_1, d_1\rangle t, ~~\forall~ t \in (0, \infty).$$ 
More generally, for any finite interval $(a,b)\subset (0,\infty)$,   
we set $c_{a,b}=S_ac_{b-a}$ and $d_{a,b}=T_ad_{b-a}$. 
Then 
\begin{eqnarray*} 
\langle c_{s_1, s_2}, d_{t_1,t_2} \rangle =
\langle c_1, d_1 \rangle |
(s_1, s_2) \cap (t_1, t_2)|.\end{eqnarray*}
\end{prop}

Clearly the real additive cocycles are the same as the real addits for the sum system $(\{G_{a,b}\},\{S^0_t\})$ 
associated with the pair $(\{S_t\},\{T_t\})$ as defined in subsection \ref{ptos}. 
We describe a precise relationship between the imaginary additive cocycles for $(\{S_t\},\{T_t\})$ and the imaginary addits for 
$(\{G_{a,b}\},\{S^0_t\})$ now. 

Let $A^\prime_t:G_{0,t} \oplus \Ran(S_t) \ni x \oplus y \mapsto x+y\in G$. 
Then it was shown in the proof of \cite[Proposition 3.6]{IS} that $A_t'$ is in $\cS(G_{0,t}\oplus \Ran(S_t),G)$. 
The following lemma is essentially the same as \cite[Remark 5.2]{IS}: 

\begin{lemma}\label{im} 
Let the notation be as above. 
\begin{itemize}
\item [$(1)$] If $\{y_t\}$ is an imaginary addit for the sum system $(\{G_{a,b}\},\{S^0_t\})$ associated with the pair 
$(\{S_t\},\{T_t\})$, then $\{({A'_t}^*)^{-1}(y_t\oplus 0)\}$ is an imaginary additive cocycle for $(\{S_t\},\{T_t\})$. 
\item [$(2)$] If $\{d_t\}$ is an imaginary additive cocycle for $(\{S_t\},\{T_t\})$, then 
${A'_t}^*d_t$ is of the form $y_t\oplus 0$ and $\{y_t\}$ is an imaginary addit for the sum system 
$(\{G_{a,b}\},\{S^0_t\})$. 
\item [(3)] For every $x_t, y_t\in G_{0,t}$, we have 
$\inpr{x_t}{({A'_t}^*)^{-1}(y_t\oplus 0)}=\inpr{x_t}{y_t}$. 
\end{itemize}
\end{lemma}

\begin{proof} (1) Assume that $\{y_t\}$ is an imaginary addit for the sum system $(\{G_{a,b}\},\{S^0_t\})$ 
and set $d_t=({A_t^\prime}^*)^{-1}(y_t\oplus 0)$. 
Then 
$$\langle d_t, S_tx\rangle=\langle y_t \oplus 0, 0 \oplus S_t x\rangle= 0,~~\forall ~ x \in G,$$ 
and hence $S_t^*({A_t^\prime}^*)^{-1}(y_t\oplus 0)=0$. 
Using $T_s = ({A_s^\prime}^*)^{-1}{A_s^\prime}^{-1}(S_s^*)^{-1}$ (see the proof of \cite[Proposition 3.6,(c)]{IS}),  
we get 
\begin{eqnarray*}
d_s+T_sd_t &=&({A'_s}^*)^{-1}(y_s\oplus 0)+({A'_s}^*)^{-1}{A'_s}^{-1}(S_s^*)^{-1}d_t\\
 &=&({A'_s}^*)^{-1}(y_s\oplus ({S_s}^*)^{-1}({A'_t}^*)^{-1}y_t).
\end{eqnarray*}
Since  we have 
\begin{eqnarray*}A'_s(I_{G_{0,s}}\oplus S_sA'_t)(f\oplus g\oplus S_th)
&=&f+S_sg+S_{s+t}h\\
&=&A'_{s+t}(A_{s,t}(I_{G_{0,s}}\oplus S_s)\oplus S_s)(f\oplus g\oplus S_th),
\end{eqnarray*}
for $f\in G_{0,s}$, $g\in G_{0,t}$, and $h\in G$, we get 
$$({A'_s}^*)^{-1}(I_{G_{0,s}}\oplus ({S_s}^*)^{-1}({A'_t}^*)^{-1})=({A'_{s+t}}^*)^{-1}(({A_{s,t}}^*)^{-1}(I_{G_{0,s}}\oplus (S_s^*)^{-1})\oplus 
({S_s}^*)^{-1}).$$
Therefore 
$$d_s+T_sd_t= ({A'_{s+t}}^*)^{-1}(({A_{s,t}}^*)^{-1}(y_s\oplus (S_s^*)^{-1}y_t)\oplus 0),$$
and the cocycle relation for $\{y_t\}$ implies $d_s+T_td_t=d_{s+t}$. 

(2) Assume conversely that $\{d_t\}$ is an imaginary additive cocycle for $(\{S_t\},\{T_t\})$. 
Then 
$$\langle {A_t^\prime}^*d_t, 0 \oplus S_tx\rangle = \langle d_t, S_t x\rangle = 0, ~~~ \forall ~ x \in G,$$ 
and hence there exists $y_t\in G_{0,t}$ such that ${A_t^\prime}^*d_t=y_t\oplus 0$. 
The above computation shows that $d_s+T_td_t=d_{s+t}$ implies 
$$({A_{s,t}}^*)^{-1}(y_s\oplus (S_s^*)^{-1}y_t)=y_{s+t},$$
and so $\{y_t\}$ is an imaginary addit for $(\{G_{a,b}\},\{S^0_t\})$. 

(3) follows from the definition of $A'_t$. 
\end{proof}

Let $A$ and $B$ be the generators of the $C_0$-semigroups $\{S_t\}$ and $\{T_t\}$ respectively. 
In what follows, we often regard $G$ as a real subspace of its complexification $G^\C$ and identify operators on $G$ 
with their complex linear extensions to $G^\C$. 
Let 
$$C=\max\{\lim_{t\to+\infty}\frac{\log ||S_t||}{t},\lim_{t\to+\infty}\frac{\log ||T_t||}{t}\},$$
which is finite thanks to \cite[page 232]{Y}. 
Then for $\RE z>C$, we have 
$$(zI-A)^{-1}f=\int_0^\infty e^{-tz}S_tfdt,$$
$$(zI-B)^{-1}f=\int_0^\infty e^{-tz}T_tfdt.$$

For a densely defined closed operator $R$ of $G$ with domain $D(R)$, we denote by $\cG(R)$ the graph of $R$ 
$$\cG(R)=\{f\oplus Rf\in G\oplus G;\; f\in D(R)\},$$
by $\inpr{\cdot}{\cdot}_R$ the graph inner product 
$$\inpr{f}{g}_R=\inpr{f}{g}+\inpr{Rf}{Rg},\quad f,g\in D(R),$$
and by $\|\cdot \|_R$ the graph norm $\|f\|=\inpr{f}{f}_R^{1/2}$. 
Then $\cG(R)$ is a closed subspace of $G\oplus G$.  
Let $\cJ$ be the unitary of $G\oplus G$ defined by 
$$\cJ(f\oplus g)=g\oplus -f,\quad f,g\in G.$$
Then $\cG(R^*)=\cJ \cG(R)^\perp$. 

The relation $T_t^*S_t=I$ implies $B\subset -A^*$ and $A\subset -B^*$. 
Let $\cK$ be the orthogonal complement of $D(B)$ in $D(-A^*)$ with respect to the inner product $\inpr{\cdot}{\cdot}_{A^*}$.   
We always regard $\cK$ as a Hilbert space equipped with the graph inner product. 
We let $\cK'$ be the orthogonal complement of $D(A)$ in $D(-B^*)$ with respect to the inner product 
$\inpr{\cdot}{\cdot}_{B^*}$ and regard it as a Hilbert space in a similar way.

\begin{lemma} \label{graph} 
The restriction of $A^*$ to $\cK$ is a unitary from $\cK$ onto $\cK'$ and it inverse is the restriction of $B^*$ to $K'$. 
\end{lemma}

\begin{proof} By definition, we have 
$$\cG(-A^*)\cap \cG(B)^\perp=\{p\oplus -A^*p;\; p\in \cK\},$$
$$\cG(-B^*)\cap \cG(A)^\perp=\{q\oplus -B^*q;\;q\in \cK'\}.$$
On the other hand, 
\begin{eqnarray*}
\cG(-B^*)\cap \cG(A)^\perp &=&\cJ (\cG(-B)^\perp\cap \cG(A^*)) \\
 &=&\{A^*p\oplus -p;\; p\in \cK\}, 
\end{eqnarray*}
which shows that $A^*|_\cK$ is a linear isomorphism from $\cK$ onto $\cK'$ whose inverse is $B^*|_{\cK'}$. 
Since 
$$\inpr{A^*f}{A^*g}_{B^*}=\inpr{A^*f}{A^*g}+\inpr{B^*A^*f}{B^*A^*g}=\inpr{f}{g}+\inpr{A^*f}{A^*g}=\inpr{f}{g}_{A^*},$$
for $f,g\in \cK$, the restriction $A^*|_\cK$ is a unitary operator. 
\end{proof}

\begin{definition} The index of a perturbation pair $(\{S_t\},\{T_t\})$ of $C_0$-semigroups is the number 
$\dim \cK=\dim \cK'$.  
\end{definition}

Recall $G_{0,t}=\Ker T_t^*$ and  $G_{0,\infty}=\overline{\cup_{t>0}G_{0,t}}$. 
For $p\in \cK$ and $t>0$, we set 
\begin{equation}\label{cocycle}
c(p)_t=\int_0^tS_spds+S_tA^*p-A^*p.
\end{equation}

\begin{lemma} \label{realadditive} 
$\{c(p)_t\}_{t>0}$ is a real additive cocycle for $(\{S_t\},\{T_t\})$. 
\end{lemma}

\begin{proof} It is easy to show the cocycle relation and all we have to show is $T_t^*c(p)_t=0$. 
Note that we have 
$$T_t^*c(p)_t=\int_0^tT_{t-s}^*pds+A^*p-T_t^*A^*p.$$
For $\RE z>2C$, 
\begin{eqnarray*}
\int_0^\infty e^{-tz}T_t^*c(p)_tdt
&=&\int_0^\infty e^{-tz}\int_0^t T_{t-s}^*p dsdt+\frac{1}{z}A^*p-\int_0^\infty e^{-tz} T_t^*A^*p dt\\
 &=& \int_0^\infty \int_s^\infty e^{-tz}T_{t-s}^*p dtds+\frac{1}{z}A^*p-(z-B^*)^{-1}A^*p\\
 &=&\frac{1}{z}(zI-B^*)^{-1}p-\frac{1}{z}(zI-B^*)^{-1}B^*A^*p\\
 &=&0,
\end{eqnarray*}
where we use Lemma \ref{graph}. 
Since this holds for all $\RE z>2C$ and the map $t\mapsto T_t^*c(p)_t$ is continuous,  we conclude 
$T_t^*c(p)_t=0$. 
\end{proof}

\begin{theorem} \label{totla} 
Let the notation be as above. 
Then
$$\overline{ \spa\{  S_ac(p)_b;\; a+b\leq t,\; p\in \cK  \} }=G_{0,t},$$ 
holds for all $t\in (0,\infty]$.  
\end{theorem}

\begin{proof} To prove the statement for $t=\infty$, it suffices to show
$$\spa\{  S_sc(p)_t;\; s,t>0,\; p\in \cK  \}^\perp\subset {G_{0,\infty}}^{\perp}.$$  
First we claim that if $f\in \spa\{c(p)_t;\; t>0,\; p\in \cK  \}^\perp$, then $(zI-A^*)^{-1}f$ belongs to the domain of $B$  
for every $\RE z>C$. 
Indeed, 
\begin{eqnarray*}0
 &=&\int_0^\infty e^{-tz}\inpr{f}{c(p)_t}dt \\
 &=&\int_0^\infty e^{-tz}\int_0^t\inpr{S_s^*f}{p}dsdt+\int_0^\infty e^{-tz}\inpr{S_t^*f}{A^*p}dt-\frac{1}{z}\inpr{f}{A^*p}\\
 &=&\frac{1}{z}\inpr{(zI-A^*)^{-1}f}{p}+\inpr{(zI-A^*)^{-1}f}{A^*p}-\frac{1}{z}\inpr{f}{A^*p}\\
 &=&\frac{1}{z}\inpr{(zI-A^*)^{-1}f}{p}+\frac{1}{z}\inpr{A^*(zI-A^*)^{-1}f}{A^*p}\\
 &=&\frac{\inpr{(zI-A^*)^{-1}f}{p}_{A^*}}{z}.
\end{eqnarray*}
This shows that $(zI-A^*)^{-1}f\in D(A^*)$ is in the orthogonal complement of $\cK$ with respect to the inner product 
$\inpr{\cdot}{\cdot}_{A^*}$, which shows the claim. 

Assume $f\in \spa\{  S_sc(p)_t;\; s,t>0,\; p\in \cK  \}^\perp$. 
Then thanks to the above claim, we have $S_t^*(zI-A^*)^{-1}f=(zI-A^*)^{-1}S_t^*f\in D(B)$ for all 
$t\geq 0$ and $\RE z>C$. 
Let $g\in D(B^*)$. 
Then the function $t\mapsto \inpr{S_t^*(zI-A^*)^{-1}f}{T_t^*g}$ is differentiable and 
\begin{eqnarray*}
\frac{d}{dt}\inpr{S_t^*(zI-A^*)^{-1}f}{T_t^*g}&=&
\inpr{A^*S_t^*(zI-A^*)^{-1}f}{T_t^*g}+\inpr{S_t^*(zI-A^*)^{-1}f}{B^*T_t^*g}\\
 &=& \inpr{A^*S_t^*(zI-A^*)^{-1}f}{T_t^*g}+\inpr{BS_t^*(zI-A^*)^{-1}f}{T_t^*g}\\
 &=&0,
\end{eqnarray*}
where we use $B\subset -A^*$. 
Therefore we get $$\inpr{T_tS_t^*(zI-A^*)^{-1}f}{g}=\inpr{(zI-A^*)^{-1}f}{g}$$ 
for all $t\geq 0$ and $g\in D(B^*)$, and so $T_tS_t^*(zI-A^*)^{-1}f=(zI-A^*)^{-1}f$. 
Thanks to ${G_{0,t}}^\perp =\Ran(T_t)$, we get $(zI-A^*)^{-1}f\in {G_{0,\infty}}^\perp$ for all $\RE z>C$ . 
Since $A^*$ is the generator of the $C_0$-semigroup $\{S_t^*\}_{t>0}$, we conclude 
$$f=\lim_{x\to+\infty}x(xI-A^{*})^{-1}f\in {G_{0,\infty}}^\perp.$$
Therefore the statement for $t=\infty$ is proven. 

Recall from Lemma \ref{direct sum} that $G=G_{0,t}+\Ran(T_t)$ is a topological direct sum such that $I-S_tT_t^*$ 
is the projection onto the first component with respect to this direct sum decomposition. 
Let 
$$G^0_{a,b}=\spa\{  S_sc(p)_t;\; a\leq s,\; s+t\leq b,\; p\in \cK  \}.$$
Then on one hand we have already seen that $G^0_{0,\infty}$ is dense in $G_{0,\infty},$ and on the other hand we have 
$(I-S_tT_t^*)G^0_{0,\infty}=G^0_{0,t}$. 
Therefore $G^0_{0,t}$ is dense in $G_{0,t}$. 
\end{proof}

By switching the roles of $\{S_t\}_{t>0}$ and $\{T_t\}_{t>0}$, 
We get the following: 

\begin{cor}\label{totali} For $q\in \cK'$ and $t>0$, we set 
$$d(q)_t=\int_0^tT_sqds+T_tB^*q-B^*q.$$
Then $\{d(q)_t\}$ is an imaginary additive cocycle for $(\{S_t\},\{T_t\})$ and  
$$\overline{ \spa\{  T_ad(q)_b;\; a+b\leq t,\; p\in \cK  \} }=\Ker(S_t^*),$$ 
holds for all $t\in (0,\infty)$.  
\end{cor}

\begin{theorem} \label{pairing} 
Let the notation be as above. Then the following holds 
$$\inpr{c(p)_t}{d(q)_t}=-t(\inpr{p}{B^*q}+\inpr{A^*p}{q}),\quad \forall p\in \cK,\; \forall q\in \cK'.$$
In consequence, we have 
$$\inpr{c({p_1})_t}{d({-A^*p_2})_t}=t\inpr{p_1}{p_2}_{A^*},\quad \forall p_1,p_2\in \cK,$$
and the maps $\cK\ni p\mapsto c(p)$ and $\cK'\ni q \mapsto d(q)$ are injective.  
\end{theorem}

\begin{proof} Direct computation shows 
\begin{eqnarray*}
\inpr{c(p)_t}{d(q)_t} &=&\int_0^t\inpr{c(p)_t}{T_sq}ds+\inpr{c(p)_t}{T_tB^*q}-\inpr{c(p)_t}{B^*q} \\
 &=& \int_0^t\inpr{c(p)_s+S_sc(p)_{t-s}}{T_sq}ds-\inpr{c(p)_t}{B^*q} \\
 &=&\int_0^t\inpr{c(p)_{t-s}}{q}ds-\inpr{c(p)_t}{B^*q}. 
\end{eqnarray*}
Thus for $\RE z>2C$, 
\begin{eqnarray*}
\int_0^\infty e^{-tz}\inpr{c(p)_t}{d(q)_t}dt
 &=& \int_0^\infty e^{-tz}\int_0^t\inpr{c(p)_{t-s}}{q}dsdt-\int_0^\infty e^{-tz}\inpr{c(p)_t}{B^*q}dt\\
 &=&\frac{1}{z}\int_0^\infty e^{-tz}\inpr{c(p)_{t}}{q}dt-\int_0^\infty e^{-tz}\inpr{c(p)_t}{B^*q}dt.
\end{eqnarray*}
On the other hand, we have 
\begin{eqnarray*}
\int_0^\infty e^{-tz}c(p)_tdt &=& \int_0^\infty e^{-tz}\int_0^t S_sp dsdt+\int_0^\infty e^{-tz}S_tA^*pdt
-\frac{1}{z}A^*p\\
 &=&\frac{1}{z}(zI-A)^{-1}p+(zI-A)^{-1}A^*p-\frac{1}{z}A^*p\\
 &=&\frac{1}{z}(zI-A)^{-1}p+\frac{1}{z}A(zI-A)^{-1}A^*p,
\end{eqnarray*}
and so 
\begin{eqnarray*}\lefteqn{
z^2\int_0^\infty e^{-tz} \inpr{c(p)_t}{d(q)_t}dt}\\ 
&=&\inpr{(zI-A)^{-1}p+A(zI-A)^{-1}A^*p}{q}-z\inpr{(zI-A)^{-1}p+A(zI-A)^{-1}A^*p}{B^*q}\\
 &=& \inpr{(zI-A)^{-1}p}{A^*B^*q}+\inpr{A(zI-A)^{-1}A^*p}{q}\\
 &+&-z\inpr{(zI-A)^{-1}p}{B^*q}-z\inpr{(zI-A)^{-1}A^*p}{A^*B^*q}\\
 &=& \inpr{A(zI-A)^{-1}p}{B^*q}+\inpr{A(zI-A)^{-1}A^*p}{q}\\
 &+&-z\inpr{(zI-A)^{-1}p}{B^*q}-z\inpr{(zI-A)^{-1}A^*p}{q}\\
& =&-\inpr{p}{B^*q}-\inpr{A^*p}{q}.
\end{eqnarray*}
Therefore the inverse Laplace transformation implies the statement. 
\end{proof}

%%%%%%%%%%%%%%%%%%%%%%%%%%%%%%%%%%%%%%%%%%%%%%%%%%%%%%%%%%%%%%%%%%%%%%%%%%%%%%%%%%%%%%%%%%%%%%%%%%%%%%%%%%%%
\subsection{Divisibility}
\begin{theorem} \label{divisible} Every sum system is divisible. 
\end{theorem}

\begin{proof} Let $(\{G_{a,b},\{S_t\}\})$ be a sum system and let $\{T_t\}$ be the $C_0$-semigroup 
constructed in subsection \ref{stop}. 
For $\{S_t\}$ and $\{T_t\}$, we use the same notation as in the previous subsection. 
Theorem \ref{totla} implies 
$$G_{0,s} =  \overline{\spa  \{x_{s_1,s_2}; (s_1,s_2) \subseteq (0,s), \{x_t\} \in\RAU\}}.$$
Lemma \ref{im} and Corollary \ref{totali} imply
$$G_{0,s}\cap\{^sy'_{a,b};\; y\in \IAU,\; 0<a,b<t\}^\perp \subset \Ker(T_s^*)\cap \Ker (S_s^*)^\perp.$$
Since $S_s$ has a left inverse $T_s^*$, the range of $S_s$ is closed, and we have $\Ker (S_s^*)^\perp=\Ran(S_s)$. 
Therefore to prove
$$G_{0,s} =  \overline{\spa_\R  \{{}^s y^\prime_{s_1,s_2}; (s_1,s_2) \subseteq(0,s), \{y_t\} \in \IAU\}},$$
it suffices to show $\Ker(T_t^*)\cap \Ran (S_t)=\{0\}$. 
Indeed, let $f\in \Ker(T_t^*)\cap \Ran (S_t)$. 
Then there exists $g\in G$ such that $f=S_tg$, and $T_t^*f=0$ on the other hand. 
Therefore $0=T_t^*S_tg=g$ and $f=0$. 
This proves the statement. 
\end{proof}

The above theorem together with Theorem \ref{IorIII} implies  

\begin{cor} \label{type} Every product system arising from a sum system is either of type I or type III. 
Every generalized CCR flow is either of type I or type III
\end{cor}

In the rest of this section, we keep employing the same assumption and notation as in the proof Theorem \ref{divisible} and 
the previous subsection. 

For $x=\{x_t\}\in \RAU$ and $y=\{y_t\}\in IAU$, we set $b_G(x,y)=\inpr{x_1}{y_1}$. 
Then Proposition \ref{h}, Theorem \ref{pairing}, and Theorem \ref{divisible} show that $b_G$ is non-degenerate as a bilinear form $b_G:\RAU\times \IAU\rightarrow \R$.   

\begin{lemma}\label{dimension} 
Let the notation be as above. Then 
$$\dim \cK=\dim \cK'=\dim \RAU=\dim \IAU.$$
\end{lemma}

\begin{proof} We already know from Lemma \ref{graph}, Theorem \ref{pairing}, and non-degeneracy of $b_G$ that 
$$\dim \cK=\dim\cK'\leq \dim \RAU=\dim \IAU$$
holds. 
Assume that $\dim \cK$ is finite. If $\dim \IAU$ were strictly larger than $\dim \cK$, there would exist 
$\{y_t\}\in \IAU\setminus \{0\}$ such that $b_G(c(p),y)=0$ for all $p\in \cK$. 
However, this contradicts Proposition \ref{h} and Theorem \ref{totla}. 
Thus we get the statement.  
\end{proof}

The above lemma allows us to introduce the index of a sum system. 

\begin{definition}\label{index} For a sum system $(\{G_{a,b}\},\{S_t\})$, 
the index $\mathrm{ind}\: G$ is the number
$$\dim \cK=\dim \cK'=\dim \RAU=\dim \IAU.$$
\end{definition}

Now we discuss an appropriate topology of $\RAU$ and $\IAU$. 
For $p\in \cK$, we set $x(p)_t=c(p)_t$ and $y(p)_t\oplus 0=(A'_t)^*d(-A^*p)_t$. 
Then $x(p)\in \RAU$, $y(p)\in \IAU$, and  $b_G(x(p_1),y(p_2))=\inpr{p_1}{p_2}_{A^*}$ thanks to Lemma \ref{im} and Theorem \ref{pairing}.

\begin{lemma}\label{topology} For the linear space $\{x(p)\in \RAU;\; p\in \cK\}$, the following three topologies coincide: 
\begin{itemize} 
\item [$(1)$] the topology of uniform convergence on every compact subset of $[0,\infty)$, 
\item [$(2)$] the topology given by the metric $d_1((x(p_1),x(p_2))=\|x(p_1)_1-x(p_2)_1\|$,
\item [$(3)$] the topology given by the metric $d_2(x(p_1),x(p_2))=\|p_1-p_2\|_{A^*}$. 
\end{itemize}
For the linear space $\{y(p)\in \IAU;\; p\in \cK\}$, the following three topologies coincide: 
\begin{itemize} 
\item [$(1)$] the topology of uniform convergence on every compact subset of $[0,\infty)$, 
\item [$(2)$] the topology given by the metric $d_1((y(p_1),y(p_2))=\|y(p_1)_1-y(p_2)_1\|$,
\item [$(3)$] the topology given by the metric $d_2(y(p_1),y(p_2))=\|p_1-p_2\|_{A^*}$. 
\end{itemize}
\end{lemma}

\begin{proof}
It is easy to show that there exists a positive constant $C_t$, increasing in $t>0$, such that 
$$\|x(p)_t\|\leq C_t\|p\|_{A^*},\quad \forall p\in \cK,$$
$$\|y(p)_t\|\leq C_t\|p\|_{A^*},\quad \forall p\in \cK.$$
Since  
$$t\|p\|^2_{A^*}=\inpr{x(p)_t}{y(p)_t}\leq \|x(p)_t\|\|y(p)_t\| \leq C_t\|p\|_{A^*}\|x(p)_t\|.$$
we get
$$\frac{t}{C_t}\|p\|_{A^*}\leq \|x(p)_t\|\leq C_t\|p\|_{A^*},$$
and in the same way, 
$$\frac{t}{C_t}\|p\|_{A^*}\leq \|y(p)_t\|\leq C_t\|p\|_{A^*}.$$
This proves the statement. 
\end{proof}

\begin{remark} The results obtained in this section substantially simplify arguments in \cite[Section 5]{IS} 
for addits in concrete examples, though the author could not find the explicit formula (\ref{cocycle}) without making 
the concrete computation in \cite[Section 5]{IS}. 
\end{remark}

%%%%%%%%%%%%%%%%%%%%%%%%%%%%%%%%%%%%%%%%%%%%%%%%%%%%%%%%%%%%%%%%%%%%%%%%%%%%%%%%%%%%%%%%%%%%%%%%%%%%%%%%%%%%%%%%%%%%%%%%%%%%%
%%%%%%%%%%%%%%%%%%%%%%%%%%%%%%%%%%%%%%%%%%%%%%%%%%%%%%%%%%%%%%%%%%%%%%%%%%%%%%%%%%%%%%%%%%%%%%%%%%%%%%%%%%%%%%%%%%%%%%%%%%%%%
%%%%%%%%%%%%%%%%%%%%%%%%%%%%%%%%%%%%%%%%%%%%%%%%%%%%%%%%%%%%%%%%%%%%%%%%%%%%%%%%%%%%%%%%%%%%%%%%%%%%%%%%%%%%%%%%%%%%%%%%%%%%%
\section{Type III criterion}

A type III criterion was obtained in \cite{IS} for a product system arising from a sum system of finite index. 
Now the description of additive cocycles in the previous section allows us to show the criterion in full generality. 

\begin{lemma}\label{Hilbert almost coboundary} Let $G$ be a real Hilbert space 
and let $(w,K)$ be a continuous unitary representation of $G$ (as an additive group) on a complex Hilbert space $K$ without containing 
the trivial representation. 
Let $c:G\rightarrow K$ be a continuous 1-cocycle, that is, the map $c$ satisfies the cocycle relation 
$$c(r+s)=c(r)+w(r)c(s),\quad \forall r,s\in G.$$
Then the following equation holds for all $r,s\in G$: 
$$\inpr{c(r)}{w(r)c(s)}=\inpr{c(s)}{w(s)c(r)}.$$
\end{lemma}

\begin{proof} Let $r,s\in G$ and let $G_0=\spa\{r,s\}$. 
Let 
$$K_0=\{f\in K;\; w(g)f=f,\;\forall g\in G_0\}$$
and let $K_1=K_0^\perp$. 
Then $K_0$ and $K_1$ are invariant under the unitary representation $w$. 
Let $c_0(g)$ and $c_1(g)$ be the projection of $c(g)$ onto $K_0$ and $K_1$ respectively. 
Then $c_0$ and $c_1$ are cocycles. 
Thanks to the cocycle relation, for every $g\in G$ and $h\in G_0$, we have 
$$c_0(g)+w(g)c_0(h)=c_0(g+h)=c_0(h)+w(h)c_0(g)=c_0(h)+c_0(g),$$
and we get $w(g)c_0(h)=c_0(h)$. 
Since $w$ does not contain the trivial representation, we have $c_0(h)=0$ for all $h\in G_0$. 
Now we can apply \cite[Lemma 4.7]{IS} to $c_1$. 
\end{proof}

Let $K$ be a complex Hilbert space. 
Recall that the automorphism group $G_K$ of the exponential product system of index $\dim K$ is described as 
follows (see \cite{Arv1}, \cite{Arv3}): 
Let $U(K)$ be the unitary group of $K$. 
Then $G_K$ is homeomorphic to $\R\times K\times U(K)$ with the group operation 
$$(a,\xi,u)\cdot(b,\eta,v)=(a+b+\IM\inpr{\xi}{u\eta},\xi+u\eta,uv). $$
For $(a,\xi,u)\in G_K$, the corresponding automorphism is realized by the family of unitary operators 
$e^{i at}W(1_{(0,t]}\xi)\EXP( 1_{(0,t]}u)$, $t>0$. 
The following is \cite[Lemma 4.8]{IS}. 

\begin{lemma}\label{gauge group} Let $G$ be an abelian group and let $\rho:G\ni r\mapsto (a(r),\xi(r),u(r))\in G_K$ be a map. 
Then $\rho$ is a homomorphism if and only if the following relation holds for every $r,s\in G$: 
\begin{equation}\label{g1}a(r+s)=a(r)+a(s)+\IM\inpr{\xi(r)}{u(r)\xi(s)},\end{equation}
\begin{equation}\label{g2} \xi(r+s)=\xi(r)+u(r)\xi(s),\end{equation}
\begin{equation}\label{g3} u(r+s)=u(r)u(s).\end{equation}
In particular, when $u(r)=1$ for all $r\in G$, then $\rho$ is a homomorphism if and only if 
\begin{equation}\label{g4} a(r+s)=a(r)+a(s),
\end{equation}
\begin{equation}\label{g5} \xi(r+s)=\xi(r)+\xi(s),
\end{equation}
\begin{equation}\label{g6} \IM \inpr{\xi(r)}{\xi(s)}=0.
\end{equation}
\end{lemma}

For a sum system $(\{G_{a,b}\}, \{S_t\})$, we denote by $\{T_t\}$ the $C_0$-semigroup constructed in subsection \ref{stop}. 
We use the same notation as in previous section such as $x(p)$, $y(p)$, and $\cK$. 
Let 
$$G^0_{0,t}=\spa  \{x(p)_{s_1,s_2}; (s_1,s_2) \subseteq
(0,t),\ p\in \cK\}\subseteq G_{0,t},$$ 
$${G^0_{0,t}}^\prime=\spa \{ {}^ty(p)^\prime_{s_1,s_2} ;(s_1,s_2) \subseteq(0,t),\;p\in \cK\}\subseteq G_{0,t}.$$ 
Then they are dense in $G_{0,t}$. 
For a subset $X\subset (0,\infty)$, we denote by $1_X$ the characteristic function of $X$.

\begin{theorem} \label{citerion} Let $(\{G_{a,b}\}, \{S_t\})$ be a sum system. 
Then the following conditions are equivalent. 
\begin{itemize}
\item [$(1)$] The product system $(H_t, U_{s,t})$ arising from $(\{G_{a,b}\}, \{S_t\})$ is of type I. 
\item [$(2)$] The sum system $(\{G_{a,b}\}, \{S_t\})$ is isomorphic to the shift sum system of index $\mathrm{ind}\; G$. 
More precisely, there exists an invertible operator $F$ from $\cK$ onto a real Hilbert space $K_\R$ such that the map 
$$G_{0,t}\ni x(p)_{a,b}\mapsto 1_{(a,b]}\otimes Fp\in L^2(0,t)_\R\otimes K_\R,\quad (a,b)\subset (0,t),\; p\in \cK.$$
extends to an isomorphism from the sum system $(\{G_{a,b}\}, \{S_t\})$ to the sum system 
$(\{L^2((a,b),K_\R)\},\{S'_t\})$, where $\{S'_t\}$ is the shift semigroup of $L^2((0,\infty),K_\R)$. 
\end{itemize}
\end{theorem}

\begin{proof} The implication from (2) to (1) is trivial. 

Assume that $(\{G_{a,b}\}, \{S_t\})$ is of type I. 
Then it is isomorphic to an exponential product system (see \cite{Arv1}, \cite{Arv3}). 
For $t >0$, let $$V_t:\Gamma(G_{0,t})\rightarrow \Gamma(L^2((0,t), K))$$ 
be a family of unitary operators, implementing the isomorphism between the 
above product systems.  Here $K$ is some separable complex Hilbert space.  

For each $p \in \cK$, the families $\{W(x(p)_t)\}_{t>0}$ and $\{W(iy(p)_t)\}_{t>0}$ form automorphisms for 
the product system  $(H_t, U_{s,t})$ (see \cite[Theorem 26]{BS}) satisfying the relations:
\begin{equation}\label{u1}
W(x(p_1)_t)W(x(p_2)_t)=W(x(p_1+p_2)_t),\quad \forall ~ p_1,p_2\in \cK,
\end{equation}
\begin{equation}\label{u2}
W(iy(p_1)_t)W(iy(p_2)_t)=W(iy(p_1+p_2)_t),\quad \forall~ p_1,p_2\in \cK, 
\end{equation}
\begin{equation}\label{u3}
W(x(p_1)_t)W(iy(p_2)_t)=e^{-2it\inpr{p_1}{p_2}}W(iy(p_2)_t)W(x(p_1)_t),\quad \forall~ p_1,p_2\in \cK.
\end{equation}
In the last equation we use $\inpr{x(p_1)_t}{y(p_2)_t}=t\inpr{p_1}{p_2}$. 
Therefore thanks to Lemma \ref{topology}, there exists two continuous homomorphisms 
$$\rho:\cK\ni p\mapsto (a(p),\xi(p),u(p))\in G_K,$$
$$\sigma:\cK\ni p\mapsto (b(p),\eta(p),v(p))\in G_K,$$
satisfying 
\begin{equation}\label{u4}
V_tW(x(p)_t)V_t^*=e^{ita(p)}W(1_{(0,t]}\xi(p))\EXP(1_{(0,t]}u(p)),\quad \forall p\in \cK,
\end{equation}
\begin{equation}\label{u5}
V_tW(iy(p)_t)V_t^*=e^{itb(p)}W(1_{(0,t]}\eta(p))\EXP(1_{(0,t]}v(p)),\quad \forall p\in \cK.
\end{equation}
Equation (\ref{u3}) implies that in addition to the relations in Lemma \ref{gauge group}, we have
\begin{equation}\label{u6}
2\inpr{p_1}{p_2}_{A^*} =\IM\inpr{\eta(p_2)}{v(p_2)\xi(p_1)}-\IM\inpr{\xi(p_1)}{u(p_1)\eta(p_2)},
\end{equation}
\begin{equation}\label{u7}\xi(p_1)+u(p_1)\eta(p_2)=\eta(p_2)+v(p_2)\xi(p_1),\end{equation}
\begin{equation}\label{u8} u(p_1)v(p_2)=v(p_2)u(p_1).\end{equation}

Let $w((p_1,p_2))=u(p_1)v(p_2)$ and $c((p_1,p_2))=\xi(p_1)+u(p_1)\eta(p_2)$. 
Then (\ref{g3}) and (\ref{u8}) imply that $(K,w)$ is a continuous unitary representation of $\cK^2$, and (\ref{g2}) and 
(\ref{u7}) imply that $c$ is a continuous $1$-cocycle. 
Let $$K_0=\{z\in K;\; w(g)z=z,\;\forall g\in \cK^2\},$$ 
and let $K_1=K_0^\perp$. 
Let $\xi_i(p)$ be the projection of $\xi(p)$ to $K_i$ and let $\eta_i(p)$ be the projection of $\eta(p)$ to $K_i$. 
Then Lemma \ref{Hilbert almost coboundary} implies 
\begin{equation}\label{u9}
\inpr{\xi_1(p_1)}{u(p_1)\eta_1(p_2)}=\inpr{\eta_1(p_2)}{v(p_2)\xi_1(p_1)}, 
\end{equation}
and Equation (\ref{u6}) is equivalent to 
\begin{equation}\label{u10}
\inpr{p_1}{p_2}_{A^*}=\IM\inpr{\eta_0(p_2)}{\xi_0(p_1)}. 
\end{equation}

We claim that $u$ and $v$ are trivial, that is, $K_1=\{0\}$.  
Assume that $K_1$ is not trivial. 
Let $0<r<s<t$. 
Then it is routine work to show  
$$V_tW(x(p)_{r,s})V_t^*=e^{i(s-r)a(p)}W(1_{(r,s]}\xi(p))\EXP(1_{(0,r]}+1_{(r,s]}u(p)+1_{(s,t]}),$$
$$V_tW(i{}^ty(p)_{r,s}')V_t^*=e^{i(s-r)b(p)}W(1_{(r,s]}\eta(p))\EXP(1_{(0,r]}+1_{(r,s]}v(p)+1_{(s,t]}).$$
By definition of $K_1$, there exists $p_0\in \cK$ such that either $u(p_0)$ or $v(p_0)$ is not trivial.  
Thus we assume that $u(p_0)\neq 1$ (the case with non-trivial $v(p_0)$ can be treated in the same way). 
Direct computation using Lemma \ref{Hilbert almost coboundary} and (\ref{g1}) shows that the operator 
$W(1_{(0,t]}\xi_1(p_0))\EXP(1_{(0,t]}u(p_0))$ commutes with $V_tW(x(p)_{r,s})V_t^*$ and $V_tW(i{}^ty(p)'_{r,s})V_t^*$ 
for all $p\in \cK$ and $0<r<s<t$. 
However, this contradicts the irreducibility of the vacuum representation of the Weyl algebra, since the sets 
$\{x(p)_{r,s}; (r,s) \subseteq (0,t),\;p\in \cK\}$ and $\{{}^ty(p)^\prime_{r,s}; (r,s) \subseteq (0,t),\; p\in \cK\}$ 
are total in $G_{0,t}$. 
Hence $K=K_0$. 

Next we claim that $\xi(\cK)+\eta(\cK)$ is dense in $K$. 
If this were not the case, there would exist a non-zero $\zeta\in K$ orthogonal to $\xi(\cK)$ and $\eta(\cK)$ with respect to 
the real inner product $\RE\inpr{\cdot}{\cdot}$. 
Again we can show that $W(i1_{(0,t]}\zeta)$ would commutes with $V_tW(x(p)_{r,s})V_t^*$ and $V_tW(i{}^ty(p)'_{r,s})V_t^*$, 
for all $p\in \cK$ and $0<r<s<t$, which is a contradiction. 
Therefore  we conclude that $\xi(\cK)+\eta(\cK)$ is dense in $K$. 

In the above argument, we have shown the following: there exist continuous homomorphisms (a priori as additive groups) 
$\xi:\cK\rightarrow K$, $\eta:\cK\rightarrow K$, $a:\cK\rightarrow \R$, and $b:\cK\rightarrow \R$ satisfying 
\begin{equation}\label{v1}
V_tW(x(p)_t)V_t^*=e^{ita(p)}W(1_{(0,t]}\xi(p)),\end{equation}
\begin{equation}\label{v2} V_tW(iy(p)_t)V_t^*=e^{itb(p)}W(1_{(0,t]}\eta(p)),\end{equation}
\begin{equation}\label{v3}
\IM\inpr{\xi(p_1)}{\xi(p_2)}=0,\quad \forall ~p_1,p_2\in \cK,
\end{equation}
\begin{equation}\label{v4}
\IM\inpr{\eta(y(p_1))}{\eta(y(p_2))}=0,\quad \forall ~p_1,p_2\in \cK,
\end{equation}
\begin{equation}\label{v5}
\inpr{p_1}{p_2}_{A^*}=\IM \inpr{\eta(p_2)}{\xi(p_1)},\quad \forall~ p_1,p_2\in \cK.  
\end{equation}
Since $\xi$, $\eta$, $a$, and $b$ are continuous, these maps are in fact linear. 

We can get rid of $a$ and $b$ in the above. 
Since $a$ and $b$ are bounded linear functionals, there exist $a_0,b_0\in \cK$ such that 
$a(p)=\inpr{p}{a_0}_{A^*}$ and $b(p)=\inpr{p}{b_0}_{A^*}$ for all $p\in \cK$. 
Let 
$$\zeta=\frac{1}{2}\xi(b_0)-\frac{1}{2}\eta(a_0).$$
Then direct computation yields $\IM \inpr{\zeta}{\xi(p)}=-a(p)/2$ and 
$\IM \inpr{\zeta}{\eta(p)}=-b(p)/2$ for all $p\in \cK$. 
By replacing $V_t$ with $W(1_{(0,t]}\zeta)V_t$, we may and do assume $a$ and $b$ are trivial. 

Equation (\ref{v5}) shows that $\xi$ and $\eta$ are injective and 
$$\|p\|_{A^*}^2=\IM \inpr{\eta(p)}{\xi(p)}\leq \|\eta(p)\| \|\xi(p)\|\leq \|\eta\|\|p\|_{A^*}\|\xi(p)\|,$$
which implies  $\|p\|_{A^*}\leq \|\eta\|\|\xi(p)\|$. 
In the same way we get $\|p\|_{A^*}\leq \|\xi\|\|\eta(p)\|$. 
In particular, the maps $\xi$ and $\eta$ are invertible from $\cK$ onto their images, which are closed. 
Thanks to Equation (\ref{v3}), the subspace $\xi(\cK)+i\xi(\cK)$ is a closed complex subspace of $K$. 
We claim $K=\xi(\cK)+i\xi(\cK)$. 
Assume $\zeta\in (\xi(\cK)+i\xi(\cK))^\perp$. 
Then since $\xi(\cK)+\eta(\cK)$ is dense in $\cK$, there exist two sequences $\{p_n\}$ and $\{q_n\}$ in $\cK$ such that 
$\{\xi(p_n)+\eta(q_n)\}$ converges to $\zeta$. 
For $p\in \cK$, we have 
$$\lim_{n\to \infty}\inpr{\xi(p_n)+\eta(q_n)}{\xi(p)}=0,\quad \forall p\in \cK.$$
This implies 
$$0=\lim_{n\to\infty}\IM \inpr{\xi(p_n)+\eta(q_n)}{\xi(p)}=\lim_{n\to\infty}\inpr{q_n}{p},$$
$$0=\lim_{n\to\infty}\RE \inpr{\xi(p_n)+\eta(q_n)}{\xi(p)}=\lim_{n\to\infty}(\inpr{\xi(p_n)}{\xi(p)}+\RE \inpr{\eta(q_n)}{\xi(p)}).$$
The first equation means that the sequence $\{q_n\}$ converges to 0 weakly, and in consequence $\{\eta(q_n)\}$ converges to 0 
weakly. 
Thus the second equation implies that the sequence $\{\xi(p_n)\}$ converges to 0 weakly as well, and so $\zeta=0$. 
This shows that $\xi(\cK)+i\xi(\cK)=K$ and we may identify $K$ with the complexification of $\xi(\cK)$. 
We denote $\xi(\cK)$ by $K_\R$. 

From now on, we regard $\xi$ as a invertible operator from $\cK$ onto the real Hilbert space $K_\R $. 
We claim that there exists a self-adjoint operator $L\in \B(K_\R )$, eventually shown to be 0, 
such that $\eta\xi^*(f)=Lf+if$ for all $f\in K_\R $. 
Since $K$ is the complexification of $K_\R $, there exist two operators $L,L'\in \B(K_\R )$ such that 
$\eta\xi^*f=Lf+iL'f$. 
For $f,g\in K_\R $, we have 
$$\IM \inpr{\eta\xi^* f}{g}=\IM\inpr{\eta\xi^*f}{\xi\xi^{-1}g}=\inpr{\xi^*f}{\xi^{-1}g}=\inpr{f}{g},$$
which shows $L'=1$. 
We also have 
$$0=\IM\inpr{\eta\xi^*f}{\eta\xi^*g}=\inpr{f}{Rg}-\inpr{Rf}{g},$$
which shows that $L$ is self-adjoint. 

Since $G^0_{0,t}+i{G^0}'_{0,t}$ is dense in $G_{0,t}^\C$, Lemma \ref{bdd} implies 
that there exists an equivalence operator $R\in \cS(G_{0,t}^\C,L^2((0,t),K))$, where $G_{0,t}^\C$ and $L^2((0,t),K)$ 
are regarded as real Hilbert spaces with real inner product $\RE \inpr{\cdot}{\cdot}$, such that for every $(r,s)\subset (0,t)$ and 
$p\in \cK$, 
$$Rx(p)_{r,s}=1_{(r,s]}\xi(p),$$
$$Ri{}^ty(p)'_{r,s}=1_{(r,s]}\eta(p).$$
We claim that $L$ is a compact operator first. 
Assume that it is not the case. 
Then there would exist a non-zero real number $\lambda$ and an orthonormal system $\{f_n\}$ in $K_\R $ such that 
$\{\|Lf_n-\lambda f_n\|\}$ converges to zero. 
We set $p_n=\xi^{-1}f_n$ and $q_n=\xi^*f_n$. 
Then $\{p_n\}$ and $\{q_n\}$ converges to 0 weakly. 
Since $\RE\inpr{x(p_n)_t}{iy(q_n)_t}=0$ and $R$ is an equivalence operator, we have 
$$\lim_{n\to\infty}\RE \inpr{Rx(p_n)_t}{Riy(q_n)_t}=0.$$
On the other hand, we have 
$$\RE \inpr{Rx(p_n)_t}{Riy(q_n)_t}=t\RE \inpr{\xi(p_n)}{\eta(q_n)}=t\RE\inpr{e_n}{Le_n}=t\lambda,$$
which is a contradiction. 

Now we show that $L=0$. 
Assume on the contrary that $L\neq 0$. 
Since $L$ is a compact self-adjoint operator, there would exist a non-zero real eigenvalue $\lambda$ with normalized 
eigenvector $f$. 
We set $p=\xi^{-1}f$ and $q=\xi^*f$. 
Let 
$$\tilde{x}_n=\sqrt{2^n}x(p)_{2^{-n}},$$
$$\tilde{y}_n=\sqrt{2^n}{}^ty(q)'_{2^{-n}}.$$
Since $\{R\tilde{x}_n\}$ and $\{Ri\tilde{y}_n\}$ are bounded, the two sequence $\{\tilde{x}_n\}$ and $\{\tilde{y}_n\}$ are bounded,   
and Proposition \ref{h}, Theorem \ref{pairing}, and Theorem \ref{divisible} show that they converge to zero weakly. 
Therefore 
$$\lim_{n\to \infty}\RE \inpr{R\tilde{x}_n}{Ri\tilde{y}_n}=0.$$
However, 
$$\RE \inpr{R\tilde{x}_n}{Ri\tilde{y}_n}=2^n\inpr{1_{(0,2^{-n}]}\otimes f}{1_{(0,2^{-n}]}\otimes Lf}=\lambda,$$
which is a contradiction. 
Therefore we get $L=0$. 
This shows that the restriction of $R$ to $G_{0,t}$ is an equivalence operator and 
finally we finishes the proof. 
\end{proof}

Let $R_t$ be the restriction of $R$ to $G_{0,t}$ in the above proof . 
Then $V_t=\Gamma(R_t)$ and 
$$Ri{}^ty(p)'_{r,s}=i(R_t^{*})^{-1}{}^ty(p)'_{r,s}.$$ 
For any positive invertible operator $D\in \B(\cK)$, we denote by 
$J^0(D)_{t}$ the linear extension of the map  
$$G^0_{0,t}\ni x(p)_{r,s}\mapsto {}^ty(Dp)'_{r,s}\in {G^0_{0,t}}'.$$
When $J^0(D)_{t,0}$ has a bounded extension to $G_{0,t}$, we denote it by $J(D)_t$. 
For $D=1$, we have 
$$i(R_t^*)^{-1}J^0(1)_tx(p)_{r,s}=1_{(r,s]}\otimes \eta(p)=1_{(r,s]}\otimes \eta\xi^*(\xi^*)^{-1}p
=i1_{(r,s]}\otimes (\xi^*)^{-1}p,$$
and so 
$$J^0(1)_{t,0}x(p)_{r,s}=R_t^*1_{(r,s]}\otimes (\xi^*)^{-1}p=R_t^*(1\otimes (\xi\xi^*)^{-1})R_tx(p)_{r,s}.$$
Therefore when the resulting product system is of type I, the operator $J^0(1)_t$ always has a bounded extension 
$J(1)_t=R_t^*(1\otimes (\xi\xi^*)^{-1})R_t$. 
A similar computation shows 
$J(\xi^*\xi)_t=R_t^*R_t$, which is an equivalence operator. 

\begin{theorem}\label{unitless} Let $(\{G_{a,b}\}, \{S_t\})$ be a 
sum system.  
Then the following statements are equivalent.
\begin{itemize}
\item[{\rm (1)}] The product system $(H_t, U_{s,t})$ arising from $(\{G_{a,b}\}, \{S_t\})$ is of type $I$.
\item[{\rm (2)}] There exists a positive invertible operator $D\in \B(\cK)$ such that 
the operator $J^0(D)_t$ extends to a bounded invertible operator $J(D)_t$ on $G_{0,t}$ for all $t>0$ 
such that $J(D)_t \in \cS(G_{0,t}, G_{0,t})$. 
\item[{\rm (3)}] There exists a positive invertible operator $D\in \B(\cK)$ such that 
the operator $J^0(D)_1$ extends to a bounded operator $J(D)_1$ on $G_{0,1}$ such that 
$J(D)_1 \in \cS(G_{0,1}, G_{0,1})$. 
\end{itemize}
\end{theorem}

\begin{proof} We have already seen that (1) implies (2). 
The implication from (2) to (3) is trivial. 
The proof of the implication from (3) to (1) is the same as the proof of \cite[Theorem 4.9]{IS}. 
\end{proof}

\begin{remark} Let the notation be as above. 
We set $F_{0,t}=\Ker(S_t^*)$ and $F_{s,t}=T_sF_{0,t-s}$. 
Then $(\{F_{a,b}\},\{T_t\})$ is also a sum system giving rise to a generalized CCR flows cocycle conjugate 
to that for $(\{G_{a,b}\},\{S_t\})$. 
When the index is one, Theorem \ref{unitless} shows that these two sum systems are isomorphic if and only if 
the resulting product system is of type I. 
This means that when the resulting product system is of type III, the two sum systems above 
are not isomorphic though they give cocycle conjugate generalized CCR flows. 
\end{remark}
\bigskip

%%%%%%%%%%%%%%%%%%%%%%%%%%%%%%%%%%%%%%%%%%%%%%%%%%%%%%%%%%%%%%%%%%%%%%%%%%%%%%%%%%%%%%%%%%%%%%%%%%%%%%%%%%%%%%%%%
%%%%%%%%%%%%%%%%%%%%%%%%%%%%%%%%%%%%%%%%%%%%%%%%%%%%%%%%%%%%%%%%%%%%%%%%%%%%%%%%%%%%%%%%%%%%%%%%%%%%%%%%%%%%%%%%%
%%%%%%%%%%%%%%%%%%%%%%%%%%%%%%%%%%%%%%%%%%%%%%%%%%%%%%%%%%%%%%%%%%%%%%%%%%%%%%%%%%%%%%%%%%%%%%%%%%%%%%%%%%%%%%%%%
\section{Perturbations of the shift}
Let $K_\R$ be a real Hilbert space, and let $\{S_t\}$ be the shift semigroup of $L^2((0,\infty),K_\R)$. 
When $\dim K_\R=1$, we gave a complete characterization of the perturbations $\{T_t\}$ of $\{S_t\}$ in terms of 
analytic functions on the right-half plane in \cite{I}, and the index of the pair $(\{S_t\},\{T_t\})$ is always 1. 
Here we take the first step to generalize our analysis to the case with non-trivial multiplicity space $K_\R$. 

\begin{theorem} Let $\{S_t\}$ be the shift semigroup of $L^2((0,\infty),K_\R)$. 
Then the index of any perturbation pair $(\{S_t\},\{T_t\})$ of $C_0$-semigroups is $\dim K_\R$. 
\end{theorem}

\begin{proof} Let $A$ and $B$ be the generators of $\{S_t\}$ and $\{T_t\}$ respectively. 
The operator $A$ is the differential operator $Af=-f'$ with the domain $D(A)$ consisting of all locally 
absolutely continuous $K_\R$-valued functions $f$ on $[0,\infty)$ such that $f'\in L^2((0,\infty),K_\R)$ and $f(0)=0$.  
The adjoint operator $A^*$ is the differential operator $Af=f'$ without boundary condition. 
Note that $B$ is a restriction of $-A^*$. 
We use the same notion as in Section 3. 
Recall that $\cK$ is the orthogonal complement of $D(B)$ in $D(A^*)$ with respect to graph inner product $\inpr{\cdot}{\cdot}_{A^*}$.  
By definition, all we have to show is $\dim \cK=\dim K_\R$. 

Let 
$$C=\lim_{t\to+\infty}\frac{\log\|T_t\|}{t}.$$
Then every real number $s>C$ belongs to the resolvent set of $B$. 
For such $s$, we introduce a linear map $M(s)$ from $D(A^*)$ to $K_\R$ by setting 
$$M(s)f=\int_0^\infty e^{-sx}f(x)dx-s\int_0^\infty e^{-sx}f'(x)dx.$$
Note that if $\xi\in K_\R$ is in the orthogonal complement of $\{M(s)p;\; p\in \cK\}$, 
then the function $e^{-sx}\xi$ belongs to the domain of $B$ and it is an eigenvector of $B$ with eigenvalue $s$. 
This implies $\dim K_\R \leq \dim \cK$. 
Indeed, if $\cK$ were finite dimensional and $\dim \cK$ were strictly smaller than $\dim K_\R$, 
every $s>C$ would belong the spectrum of $B$, which is a contradiction. 

Suppose now that $\dim \cK$ is strictly larger than $\dim K_\R$. 
We fix $s>C$.  
Then there would exist $p\in \cK\setminus \{0\}$ such that $M(s)p$=0. 
Take $f\in L^2((0,\infty),K_\R)$ and set $g=(sI-B)^{-1}f$. 
Solving the differential equation $g'+sg=f$, we get 
$$g(x)=e^{-sx}g(0)+\int_0^x e^{s(t-x)}f(t)dt. $$
Since $g\in D(B)$, we have 
\begin{eqnarray*}
0&=&\inpr{p}{g}+\inpr{p'}{g'}=\inpr{p}{g}+\inpr{p'}{f-sg}
=\inpr{p'}{f}+\inpr{p-sp'}{g}\\
&=&\inpr{p'}{f}+\int_0^\infty\inpr{p(x)-sp'(x)}{e^{-sx}g(0)+\int_0^x e^{s(t-x)}f(t)dt}dx\\
&=&\inpr{p'}{f}+\inpr{M(s)p}{g(0)}+\int_0^\infty \int_0^x e^{s(t-x)}\inpr{p(x)-sp'(x)}{f(t)}dtdx\\
&=&\inpr{p'}{f}+\int_0^\infty \int_0^\infty e^{-sy}\inpr{p(y+t)-sp'(y+t)}{f(t)}dydt\\
&=&\inpr{p'}{f}+\inpr{(sI-A^*)^{-1}(p-sp')}{f}.
\end{eqnarray*}
Since $f$ is arbitrary, we get $p'+(sI-A^*)^{-1}(p-sp')=0$. 
This implies $p'\in D(A^*)$ and $sp'-p''+p-sp'=0$, and so $p''=p$. 
Since $p\in L^2((0,\infty),K_\R)$ and $p\neq 0$, this is possible only if $p(x)=e^{-x}\xi$ for some $\xi\in K_\R\setminus\{0\}.$ 
However, this $p$ does not satisfy  $M(s)p=0$, which is a contradiction, and we conclude $\dim \cK=\dim K_\R$. 
\end{proof}

%%%%%%%%%%%%%%%%%%%%%%%%%%%%%%%%%%%%%%%%%%%%%%%%%%%%%%%%%%%%%%%
%%%%%%%%%%%%%%%%%%%%%%%%%%%%%%%%%%%%%%%%%%%%%%%%%%%%%%%%%%%%%%%
%%%%%%%%%%%%%%%%%%%%%%%%%%%%%%%%%%%%%%%%%%%%%%%%%%%%%%%%%%%%%%%

%\bibitem{Z} K. H. Zhu, 
%\textit{Operator Theory in Function Spaces.} 
%Monographs and Textbooks in Pure and Applied Mathematics, 139. Marcel Dekker, Inc., 
%New York, 1990.

\thebibliography{99}

%\bibitem{Ara1} H. Araki,  
%\textit{A lattice of von Neumann algebras associated with the quantum theory of a free Bose field.} 
%J. Math. Phys. \textbf{4} (1963), 1343--1362. 

%\bibitem{Ara2} H. Araki,  
%\textit{Type of von Neumann algebra associated with free Bose field,}
%Progr. Theoret. Phys. \textbf{32}, (1964), 956--965. 


\bibitem{Ara3} H. Araki, 
\textit{On quasifree states of the canonical commutation relations. II.} 
Publ. Res. Inst. Math. Sci. \textbf{7} (1971/72), 121--152. 


\bibitem{Arv1} W. Arveson, 
\textit{Continuous analogues of Fock spaces}, Mem. Americ. Math. Soc. 80(409):1-66, 1989. 

\bibitem{Arv2} W. Arveson, {\it Continuous analogues of Fock
spaces IV  Essential states,} Acta Math. 164 (3/4) 265-300, 1990. 
  
\bibitem{Arv3} W. Arveson, 
\textit{Non-commutative dynamics and $E$-semigroups}, Springer Monograph in Math. (Springer 2003).


\bibitem{BS} B. V. R. Bhat and R. Srinivasan, \textit{On product systems arising from sum systems} 
Infinite dimensional analysis and related topics, 
Vol. 8, Number 1, March 2005.


\bibitem{vD} A. van Daele,  
\textit{Quasi-equivalence of quasi-free states on the Weyl algebra.} 
Comm. Math. Phys. \textbf{21} (1971), 171--191.

%\bibitem {F} W. Feller, 
%\textit{An Introduction to Probability Theory and its Applications. Vol. II.} 
%Second edition, John Wiley \& Sons, Inc., New York-London-Sydney 1971. 

%\bibitem{H} K. Hoffman,  
%\textit{Banach Spaces of Analytic Functions.} 
%Prentice-Hall, Inc., Englewood Cliffs, N. J. 1962.

\bibitem{I} M. Izumi, 
\textit{A perturbation problem for the shift semigroup.} Preprint math/0702439, 2007, 
to appear in J. Funct. Anal. 

\bibitem{IS} M. Izumi and R. Srinivasan, 
\textit{Generalized CCR flows.} Preprint math/0705.3280, 2007, to appear in Commun. Math. Phys. 

%\bibitem{K} P. Koosis, 
%\textit{Introduction to $H^p$ Spaces.} Second edition. With two appendices by V. P. Havin. 
%Cambridge Tracts in Mathematics, 115. Cambridge University Press, Cambridge, 1998.


\bibitem{L} V. Liebscher, \textit{Random sets and invariants for
(type $II$) continuous product systems of Hilbert spaces},
Preprint math.PR/0306365.

\bibitem{Pa} K. R. Parthasarathy, \textit{An Introduction to Quantum
Stochastic Calculus}, Birkauser Basel, Boston, Berlin (1992). 

\bibitem{Po1} R. T. Powers, \textit{A nonspatial continuous semigroup of $*$-endomorphisms
of $B(H)$,} Publ. Res. Inst. Math. Sci.  23 (1987), 1053-1069.

%\bibitem{Po2} R.T. Powers, {\it New examples of continuous spatial semigroups of endomorphisms of $B(H)$,} Int. J. Math. 10 (2):215-288, (1999).

%\bibitem{Pr} G. L. Price, B. M. Baker, P. E. T. Jorgensen and P. S. Muhly
%(Editors), {\it Advances in Quantum Dynamics}, 
%(South Hadley, MA, 2002) Contemp.
%Math. 335, Amer. Math. Society, Providence, RI (2003).

%\bibitem{S} Y. Shalom, 
%\textit{Harmonic analysis, cohomology, and the large-scale geometry of amenable groups.}
%Acta Math. \textbf{192} (2004), 119--185.

\bibitem{Sk} M. Skeide, 
\textit{Existence of $E_0$-semigroups for Arveson systems: making two proofs into one.} 
Infin. Dimens. Anal. Quantum Probab. Relat. Top. \textbf{9} (2006), 373--378.

\bibitem{T1} B. Tsirelson, 
\textit{Non-isomorphic product systems.} 
Advances in Quantum Dynamics (South Hadley, MA, 2002), 273--328, 
Contemp. Math., 335, Amer. Math. Soc., Providence, RI, 2003. 

%\bibitem{T2} B. Tsirelson,  
%\textit{Spectral densities describing off-white noises.} 
%Ann. Inst. H. Poincar\'e Probab. Statist. \textbf{38} 
%(2002), 1059--1069.


\bibitem{Y} K. Yosida,  
\textit{Functional Analysis.} 
Sixth edition. Springer-Verlag, Berlin-New York, 1980.
\end{document}